\documentclass[12pt]{article} 
\usepackage{amsfonts,amsmath,latexsym,amssymb,mathrsfs,amsthm,comment,setspace}
\usepackage{slashbox}
\usepackage{caption}

\let\OLDthebibliography\thebibliography
\renewcommand\thebibliography[1]{
  \OLDthebibliography{#1}
  \setlength{\parskip}{0pt}
  \setlength{\itemsep}{0pt plus 0.0ex}
}

\def\numberlikeadb{\global\def\theequation{\thesection.\arabic{equation}}}
\numberlikeadb
\newtheorem{theorem}{Theorem}[section]
\newtheorem{lemma}[theorem]{Lemma}
\newtheorem{corollary}[theorem]{Corollary}

\newtheorem{remark}[theorem]{Remark}

\usepackage{color}

\usepackage{lscape}
\usepackage{caption}
\usepackage{multirow}
\begin{document}

\title{Bounds for the chi-square approximation of the power divergence family of statistics}
\author{Robert E. Gaunt\footnote{Department of Mathematics, The University of Manchester, Oxford Road, Manchester M13 9PL, UK, robert.gaunt@manchester.ac.uk}}  


\date{} 
\maketitle

\vspace{-5mm}


\begin{abstract}It is well-known that each statistic in the family of power divergence statistics, across $n$ trials and $r$ classifications with index parameter $\lambda\in\mathbb{R}$ (the Pearson, likelihood ratio and Freeman-Tukey statistics correspond to $\lambda=1,0,-1/2$, respectively) is asymptotically chi-square distributed as the sample size tends to infinity. In this paper, we obtain explicit bounds on this distributional approximation, measured using smooth test functions, that hold for a given finite sample $n$, and all index parameters ($\lambda>-1$) for which such finite sample bounds are meaningful. We obtain  bounds that are of the optimal order $n^{-1}$. The dependence of our bounds on the index parameter $\lambda$ and the cell classification probabilities is also optimal, and the dependence on the number of cells is also respectable.  Our bounds generalise, complement and improve on recent results from the literature.
\end{abstract}


\noindent{{{\bf{AMS 2010 Subject Classification:}}} Primary 60F05; 62E17

\section{Introduction}


Consider the setting of a multinomial goodness-of-fit test, with $n$ independent trials in which each trial leads to a unique classification over $r\geq2$ classes. Let $U_1,\ldots,U_r$ represent the observed frequencies arising in each class,
 and denote the non-zero classification probabilities by $\boldsymbol{p}=(p_1,\ldots,p_r)$. In 1984, Cressie and Read \cite{cr84} introduced the \emph{power divergence family of statistics} for testing the null hypothesis $H_0$: $\boldsymbol{p}=\boldsymbol{p}_0$ against the alternative hypothesis $H_1$: $\boldsymbol{p}=\boldsymbol{p}_1$. These statistics are given by
\begin{equation}\label{powerdiv}T_\lambda=\frac{2}{\lambda(\lambda+1)}\sum_{j=1}^rU_j\bigg[\bigg(\frac{U_j}{np_j}\bigg)^\lambda-1\bigg],
\end{equation}
where the index parameter $\lambda\in\mathbb{R}$. Here, and throughout the paper, we assume the validity of the null hypothesis, and, for ease of notation, suppress the subscript in the notation $\boldsymbol{p}_0$. Note that $\sum_{j=1}^rU_j=n$ and $\sum_{j=1}^rp_j=1$.

When $\lambda=0,-1$, the notation (\ref{powerdiv}) should be understood as a result of passage to the limit, in which cases we recover the log-likelihood ratio and modified log likelihood ratio statistics, respectively:
\begin{align}\label{lrs} L&=2\sum_{j=1}^r U_j\log\bigg(\frac{U_j}{np_j}\bigg),\\
\label{klstat}GM^2&=2n\sum_{j=1}^r p_j\log\bigg(\frac{np_j}{U_j}\bigg).
\end{align}
The cases $\lambda=1,-1/2$ correspond to Pearson's chi-square statistic \cite{pearson} and the Freeman-Tukey statistic \cite{ft50}, given by
\begin{align} \label{pearsonstatistic}\chi^2 &= \sum_{j=1}^r \frac{(U_j - n p_j)^2}{n p_j}, \\
T^2&=4\sum_{j=1}^r\big(\sqrt{U_j}-\sqrt{np_j}\big)^2. \nonumber
\end{align}
The power divergence family of statistics can therefore be seen to unify several commonly used goodness-of-fit tests. The statistic $T_{2/3}$, often referred to as the Cressie-Read statistic, was also suggested by \cite{cr84,rc88} as a good alternative to the classical likelihood ratio and Pearson statistics. 

A fundamental result 
is that, 
for all $\lambda\in\mathbb{R}$, the statistic $T_\lambda$ converges in distribution to the $\chi_{(r-1)}^2$ distribution as the number of trials $n$ tends to infinity (see \cite{cr84}, p.\ 443). However, in practice one has a finite sample, and so it is of interest to assess the quality of this distributional approximation for finite $n$.  In the literature, this has been done on the one hand through theoretical bounds  on the rate of convergence, which will be the focus of this paper. It has also been done through simulation studies. The findings of simulation studies from the literature are briefly reviewed in Remark \ref{remsimul}, in which we also discuss how the theoretical bounds obtained in this paper reflect the findings of these studies.

 From the theoretical literature, the quality of the chi-square approximation of the distibution of the statistic $T_{\lambda}$ has been assessed by Edgeworth expansions \cite{a10,asylbekov,fv20,pu13,read84,ulyanov}. In \cite{read84}, Edgeworth expansions were used to propose closer approximations to the distribution function of $T_{\lambda}$ that are of the form of the $\chi_{(r-1)}^2$ distribution function plus a $o(1)$ correction. The work of \cite{read84} generalised results of \cite{sf84} that had been given for the Pearson, likelihood ratio and Freeman-Tukey statistics ($\lambda=1,0,-1/2$).  However, as noted by \cite{pu13}, it is impossible to determine a rate of convergence from the expansions given in \cite{read84,sf84}. Bounds on the rate of convergence of the distribution of $T_\lambda$ to its limiting $\chi_{(r-1)}^2$ distribution via Edgeworth expansions were obtained by \cite{fv20,pu13}. For $r\geq4$,  \cite{ulyanov} obtained a $O(n^{-(r-1)/r})$ bound on the rate of convergence in the Kolmogorov distance (see also \cite{a10} for a refinement of this result) and, for $r=3$, \cite{asylbekov} obtained a $O(n^{-3/4+0.065})$ bound on the rate of convergence in the same metric, 
  with both bounds holding for all $\lambda\in\mathbb{R}$. 

Bounds on the rate of convergence have also been given for special cases of the power divergence family. For the likelihood ratio statistic, \cite{ar20} recently used Stein's method to obtain an explicit $O(n^{-1/2})$ bound for smooth test functions. It is worth noting that the setting of \cite{ar20} was more general than the categorical data setting of this paper. For Pearson's statistic, it was shown by \cite{yarnold} using Edgeworth expansions that a bound on the rate of convergence of Pearson's statistic, in the Kolmogorov distance is $O(n^{-(r-1)/r})$, for $r\geq2$, which was improved by \cite{gu03} to $O(n^{-1})$ for $r\geq6$. An explicit $O(n^{-1/2})$ Kolmogorov distance bound was obtained by Stein's method in \cite{mann97}. Recently, \cite{gaunt chi square} used Stein's method to obtain explicit error bounds for Pearson's statistic, measured using smooth test functions. These bounds will be needed in the sequel, and are recorded here.

Let $C_b^k(\mathbb{R}^+)$ denote the class of bounded
functions $h : \mathbb{R}^+\rightarrow\mathbb{R}$ for which $h^{(k)}$ exists and derivatives up to $k$-th order are bounded. Let $r\geq2$ and $p_*:=\min_{1\leq j\leq r}p_j$. Suppose $np_*\geq 1$.
 Then, the following bounds were obtained in \cite{gaunt chi square}. For $h \in C_b^5(\mathbb{R}^+)$, 
\begin{align}\label{pearbound1}
|\mathbb{E} [h(\chi^2)] - \chi^2_{(r-1)} h | &\leq \frac{4}{(r+1)n}\bigg(\sum_{j=1}^r\frac{1}{\sqrt{p_j}}\bigg)^2\{19\|h\|+366\|h'\|+2016\|h''\|\nonumber\\
&\quad+5264\|h^{(3)}\|+106\,965\|h^{(4)}\|+302\,922\|h^{(5)}\|\},
\end{align}
and, for $h \in C_b^2(\mathbb{R}^+)$,
\begin{equation}\label{pearbound01}
|\mathbb{E} [h(\chi^2)] - \chi^2_{(r-1)} h | \leq  \frac{24}{(r+1)\sqrt{n}}\{3\|h\|+23\|h'\|+42\|h''\|\}\sum_{j=1}^r\frac{1}{\sqrt{p_j}},
\end{equation}
where $\chi_{(r-1)}^2h$ denotes the expectation $\mathbb{E}[h(Y_{r-1})]$ for $Y_{r-1}\sim \chi_{(r-1)}^2$. Here and throughout the paper, we write $\|\cdot\|$ for the usual supremum norm $\|\cdot\|_\infty$ of a real-valued function. The bound (\ref{pearbound1}) achieves the optimal $O(n^{-1})$ rate, although when the constants are large compared to $n$ the bound (\ref{pearbound01}) may give the
smaller numerical bound. The bounds (\ref{pearbound1}) and (\ref{pearbound01}) have an optimal dependence on $p_*$, since, for a fixed number of classes $r$, both bounds tend to zero if and only if $np_*\rightarrow\infty$, which is precisely the condition under which $\chi^2\rightarrow_d \chi_{(r-1)}^2$ (see \cite{gn96}). The bounds (\ref{pearbound1}) and (\ref{pearbound01}) are the first in the literature on Pearson's statistic to decay to zero if and only if $np_*\rightarrow\infty$, and in the case of (\ref{pearbound1}) achieve the optimal $O(n^{-1})$ rate for all $r\geq2$. 


In this paper, we generalise 
the bounds of \cite{gaunt chi square} to the family of power divergence statistics $T_\lambda$, $\lambda>-1$, the largest subclass of the power divergence family for which finite sample bounds can be given when measured using smooth test functions, as in (\ref{pearbound1}) and (\ref{pearbound01}). 
Indeed, it can be seen from (\ref{powerdiv}) and (\ref{klstat}) that, for $\lambda\leq-1$, we have $T_\lambda=\infty$ if any of the observed frequencies $U_1,\ldots,U_r$ are zero, meaning that the expectation $\mathbb{E}[h(T_\lambda)]$ is undefined. The bounds in Theorem \ref{thm1} are of the optimal $n^{-1}$ order. Specialising to the case $\lambda=0$ of the likelihood ratio statistic, our $O(n^{-1})$ bound improves on the $O(n^{-1/2})$ rate (and has a better dependence on the other model parameters) of \cite{ar20} that was given in a more general setting than that of categorical data considered in this paper. 
Our results also complement a work in preparation of \cite{gr friedman} in which Stein's method is used to obtained order $O(n^{-1})$ bounds for the chi-square approximation of Friedman's statistic.  
In Theorem \ref{thm2}, we provide sub-optimal $O(n^{-1/2})$ bounds which may yield smaller numerical bounds for small sample sizes $n$. 

Like the bounds of \cite{gaunt chi square} for Pearson's statistic, for all fixed $\lambda>-1$ and $r\geq2$, the bounds of Theorems \ref{thm1} and \ref{thm2} enjoy the property of decaying to zero if and only if $np_*\rightarrow\infty$. If $r\geq2$ is fixed and we allow $\lambda$ to vary with $n$, then the bounds of Theorems \ref{thm1} and \ref{thm2} tend to zero if and only if $np_*/\lambda^2\rightarrow\infty$, which is again an optimal condition (see Remark \ref{remconds}). The dependence of our bounds on $r$ is also respectable.   The excellent dependence of the bounds of Theorem \ref{thm2} on all parameters allows us to prescribe simple conditions under which the chi-square approximation is valid, when the quantities $\lambda$, $r$ and $p_1,\ldots,p_r$ may vary with $n$; see Remark \ref{remconds}. As already discussed, some of these conditions are optimal, and in parameter regimes in which we have not been able to obtain optimal conditions we conjecture what these optimal conditions are; see Remark \ref{remconds}, again. 

It is perhaps not a priori obvious that the $O(n^{-1})$ rate of our main bounds would hold for all power divergence statistics with $\lambda>-1$. For example, simulation studies have shown that for small samples $n$ the likelihood ratio statistic ($\lambda=0$) is less accurate at approximating the limiting $\chi_{(r-1)}^2$ distribution than is Pearson's statistic ($\lambda=1$) (see, for example, \cite{c76,l78}). Our results show, however, that at least when measured using smooth test functions, the actual rate of convergence is the same for large $n$ if the number of classes $r$ is fixed. If the number of classes $r$ is allowed to grow with $n$, then Pearson's statistic will converge faster, though (see Remark \ref{remconds}). It is also important to note that our assumption of smooth test functions is essential for the purpose of obtaining $O(n^{-1})$ convergence rates that hold for all $r\geq2$ and all $\lambda>-1$. For example, in the case $r=3$, convergence rates in the Kolmogorov metric can be no faster than $O(n^{-3/4}\log\log n)$; see \cite[Remark 3]{asylbekov}.

The rest of the paper is organised as follows. In Section \ref{sec2}, we state our main results. Theorem \ref{thm0} provides analogues of the bounds (\ref{pearbound1}) and (\ref{pearbound01}) for Pearson's statistic that hold for wider classes of test functions at the expense of a worse dependence on the parameter $r$. Our main result, Theorem \ref{thm1}, provides $O(n^{-1})$ bounds to quantify the chi-square approximation of the power divergence statistics. Alternative $O(n^{-1/2})$ bounds are given in Theorem \ref{thm2}, and from one of these bounds we extract a Kolmogorov distance bound (Corollary \ref{cor1}). We end Section \ref{sec2} with several remarks discussing the bounds. In Section \ref{sec3}, we present several preliminary lemmas that are needed for the proofs of our main results. We prove the main results in Section \ref{sec4}. In Appendix \ref{appchi}, we present some basics of Stein's method for chi-square approximation and use it to give a short proof of Theorem \ref{thm0}. Finally, in Appendix \ref{appa}, we prove a technical lemma from Section \ref{sec3}.

\section{Main results}\label{sec2}

We first present a theorem which complements the main result of \cite{gaunt chi square} by giving explicit error bounds for the chi-square approximation of Pearson's statistic that hold for larger classes of function than those used by \cite{gaunt chi square}. The proof involves a minor modification of the proof of the bounds (\ref{pearbound1}) and (\ref{pearbound01}) of \cite{gaunt chi square} and the details are given in Appendix \ref{appchi}.

Before stating the theorem, we introduce some notation. We let $p_*:=\mathrm{min}_{1\leq j\leq r}p_j$. Let $C_b^{j,k}(\mathbb{R}^+)$, $j\leq k$, denote the class of
functions $h : \mathbb{R}^+\rightarrow\mathbb{R}$ for which $h^{(k)}$ exists and derivatives of order $j,j+1,\ldots,k$ are bounded. Note that $C_b^k(\mathbb{R}^+)\subset C_b^{j,k}(\mathbb{R}^+)$ for $j\geq1$.


\begin{theorem}\label{thm0}Let $(U_1, \ldots, U_r)$ be the multinomial vector of $n$ observed counts, where $r \geq 2$, and suppose that $np_*\geq 1$. Let $\chi^2$ be Pearson's chi-square statistic as defined in (\ref{pearsonstatistic}).  Then, for $h\in C_b^{1,5}(\mathbb{R}^+)$, 
\begin{align}\label{pearbound1bb}
|\mathbb{E} [h(\chi^2)] - \chi^2_{(r-1)} h | &\leq \frac{1}{(r+1)^{1/2}n}\bigg(\sum_{j=1}^r\frac{1}{\sqrt{p_j}}\bigg)^2\{122\|h'\|+1970\|h''\|\nonumber\\
&\quad+6943\|h^{(3)}\|+12\,731\|h^{(4)}\|+643\,710\|h^{(5)}\|\},
\end{align}
and, for $h\in C_b^{1,2}(\mathbb{R}^+)$,
\begin{equation}\label{pearbound01bb}
|\mathbb{E} [h(\chi^2)] - \chi^2_{(r-1)} h | \leq  \frac{1}{\sqrt{(r+1)n}}\{115\|h'\|+536\|h''\|\}\sum_{j=1}^r\frac{1}{\sqrt{p_j}}.
\end{equation}
Also, for $h\in C_b^{2,5}(\mathbb{R}^+)$,
\begin{align}\label{pearbound1cc}
|\mathbb{E} [h(\chi^2)] - \chi^2_{(r-1)} h | &\leq \frac{1}{n}\bigg(\sum_{j=1}^r\frac{1}{\sqrt{p_j}}\bigg)^2\{19\|h''\|+206\|h^{(3)}\|\nonumber\\
&\quad+545\|h^{(4)}\|+161\,348\|h^{(5)}\|\},
\end{align}
and, for $h \in C_b^{2,2}(\mathbb{R}^+)$,
\begin{equation}\label{pearbound01cc}
|\mathbb{E} [h(\chi^2)] - \chi^2_{(r-1)} h | \leq  \frac{238\|h''\|}{\sqrt{n}}\sum_{j=1}^r\frac{1}{\sqrt{p_j}}.
\end{equation}
\end{theorem}

The following weak convergence theorem for smooth test functions with a bound of order $n^{-1}$ for the $\chi_{(r-1)}^2$ approximation of the power divergence statistic $T_\lambda$ (which holds for all $\lambda>-1$ and $r\geq2$) is the main result of this paper.

\begin{theorem}\label{thm1}Let $(U_1, \ldots, U_r)$ represent the multinomial vector of $n$ observed counts, where $r \geq 2$. For $\lambda>-1$, let $T_\lambda$ be the power divergence statistic as defined in (\ref{powerdiv}). Suppose that $np_*\geq r$. If $\lambda\geq3$, we also suppose that $np_*\geq 2(\lambda-3)^2$. Then, for $h \in C_b^5(\mathbb{R}^+)$,
\begin{align}\label{pdbound}
|\mathbb{E} [h(T_\lambda)] - \chi^2_{(r-1)} h | &\leq \frac{1}{n}\bigg(\sum_{j=1}^r\frac{1}{p_j}\bigg)\bigg\{\frac{4r}{r+1}\{19\|h\|+366\|h'\|+2016\|h''\|\nonumber\\
&\quad+5264\|h^{(3)}\|+106\,965\|h^{(4)}\|+302\,922\|h^{(5)}\|\}\nonumber\\
&\quad+|\lambda-1|r\{2\|h'\|+202\|h''\|+819\|h^{(3)}\|+100\,974\|h^{(4)}\|\}\nonumber\\
&\quad+\frac{19}{9}(\lambda-1)^2\|h''\|+\frac{|(\lambda-1)(\lambda-2)|(12\lambda+13)}{6(\lambda+1)}\|h'\|\bigg\}.
\end{align}

Also, under the weaker assumption that $h\in C_b^{1,5}(\mathbb{R}^+)$,
\begin{align}\label{pdbound22}
|\mathbb{E} [h(T_\lambda)] - \chi^2_{(r-1)} h | &\leq \frac{1}{n}\bigg(\sum_{j=1}^r\frac{1}{p_j}\bigg)\bigg\{\frac{r}{\sqrt{r+1}}\{122\|h'\|+1970\|h''\|\nonumber\\
&\quad+6943\|h^{(3)}\|+12\,731\|h^{(4)}\|+643\,710\|h^{(5)}\|\}\nonumber\\
&\quad+|\lambda-1|r\{2\|h'\|+202\|h''\|+819\|h^{(3)}\|+100\,974\|h^{(4)}\|\}\nonumber\\
&\quad+\frac{19}{9}(\lambda-1)^2\|h''\|+\frac{|(\lambda-1)(\lambda-2)|(12\lambda+13)}{6(\lambda+1)}\|h'\|\bigg\}.
\end{align}
\end{theorem}

If the constants are large compared to $n$, the next result may give smaller numerical bounds. 

\begin{theorem}\label{thm2}Let $\lambda>-1$, $r\geq2$ and $np_*\geq 2$.  If $\lambda\geq2$, we also suppose that $np_*\geq 2(\lambda-2)^2$.  Then, for $h \in C_b^2(\mathbb{R}^+)$,
\begin{align}
&|\mathbb{E} [h(T_\lambda)] - \chi^2_{(r-1)} h | \nonumber\\
\label{pdbound0}&\quad\leq \bigg(\frac{24}{r+1}\{3\|h\|+23\|h'\|+42\|h''\|\}+\frac{|\lambda-1|(4\lambda+7)}{\lambda+1}\|h'\|\bigg)\sum_{j=1}^r\frac{1}{\sqrt{np_j}}.
\end{align}
Also, under the weaker assumption that $h\in C_b^{1,2}(\mathbb{R}^+)$,
\begin{align}
&|\mathbb{E} [h(T_\lambda)] - \chi^2_{(r-1)} h | \nonumber\\
\label{pdbound022}&\quad\leq \bigg(\frac{1}{\sqrt{r+1}}\{115\|h'\|+536\|h''\|\}+\frac{|\lambda-1|(4\lambda+7)}{\lambda+1}\|h'\|\bigg)\sum_{j=1}^r\frac{1}{\sqrt{np_j}}.
\end{align}
\end{theorem}

In the following corollary we deduce a Kolmogorov distance bound from (\ref{pdbound0}) by applying a basic technique (see \cite[p.\ 48]{chen}) for converting smooth test function bounds into Kolmogorov distance bounds. As this technique is fairly crude, our resulting bound has a suboptimal dependence on $n$, although it does inherit the desirable property of decaying to zero if and only if $np_*\rightarrow\infty$, for fixed $r$ and $\lambda$. This is the first Kolmogorov distance bound for the chi-square approximation of the statistic $T_\lambda$ to enjoy this property.


\begin{corollary}\label{cor1} Let $\lambda>-1$, $r\geq2$ and $np_*\geq 2$.  If $\lambda\geq2$, we additionally suppose that $np_*\geq 2(\lambda-2)^2$. Also, let $Y_{r-1}\sim\chi_{(r-1)}^2$.  Then
\begin{align*}&\sup_{z>0}|\mathbb{P}(T_\lambda\leq z)-\mathbb{P}(Y_{r-1}\leq z)| \\
&\leq  \begin{cases} \displaystyle \frac{1}{(np_*)^{1/10}}\bigg\{8+\bigg(21+\frac{|\lambda-1|(4\lambda+7)r}{52(\lambda+1)}\bigg)\frac{1}{(np_*)^{1/5}}+\frac{72}{(np_*)^{2/5}}\bigg\}, & \:  r=2, \\
\displaystyle \frac{1}{(np_*)^{1/6}}\bigg\{19+\bigg(44+\frac{|\lambda-1|(4\lambda+7)r}{25(\lambda+1)}\bigg)\frac{1}{(np_*)^{1/6}}+\frac{72}{(np_*)^{1/3}}\bigg\}, & \:  r=3, \\
\displaystyle \frac{1}{(r-3)^{1/3}(np_*)^{1/6}}\bigg\{13+\bigg(37+\frac{|\lambda-1|(4\lambda+7)r}{30(\lambda+1)}\bigg)\frac{(r-3)^{1/6}}{(np_*)^{1/6}}\\
\displaystyle+\frac{72(r-3)^{1/3}}{(np_*)^{1/3}}\bigg\}, & \:   r\geq4. \end{cases}
\end{align*}
\end{corollary}



\begin{remark}[On the proofs and assumptions]
Our proofs of Theorems \ref{thm1} and \ref{thm2} involve writing the power divergence statistic $T_\lambda$ in the form $T_\lambda=\chi^2+R$, where $R$ is a ``small" remainder term. By this approach, to obtain our bounds on the quantity of interest $|\mathbb{E}[h(T_\lambda)]-\chi_{(r-1)}^2h|$, we make use of the bounds on $|\mathbb{E}[h(\chi^2)]-\chi_{(r-1)}^2h|$ given by (\ref{pearbound1}), (\ref{pearbound01}), (\ref{pearbound1bb}) and (\ref{pearbound01bb}). In the case $\lambda=1$, $T_\lambda=\chi^2$, meaning that $R=0$, which explains why our bounds reduce almost exactly to those bounds when $\lambda=1$, with the only difference being that at the end of our proof of Theorem \ref{thm1} we use the inequality $\big(\sum_{j=1}^rp_j^{-1/2}\big)^2\leq r\sum_{j=1}^rp_j^{-1}$ to obtain a compact final bound.  

Our approach of decomposing $T_\lambda=\chi^2+R$ means that we inherit large numerical constants from the bounds (\ref{pearbound1}), (\ref{pearbound01}), (\ref{pearbound1bb}) and (\ref{pearbound01bb}). However, as was the case in \cite{gaunt chi square}, our primary concern is simple final bounds with good dependence on the parameters $n$, $\lambda$, $r$ and $p_1,\ldots,p_r$. In proving Theorems \ref{thm1} and \ref{thm2}, we make some crude approximations but take care to ensure they do not affect the role of the parameters in the final bound.

In order to simplify the calculations and arrive at our desired compact final bounds, we made some assumptions. We discuss the assumptions for Theorem \ref{thm1}; similar comments apply to the assumptions of Theorem \ref{thm2}. The assumption $np_*\geq r$ ($\lambda\not=1$), $np_*\geq1$ ($\lambda=1$) is very mild (the bounds in Theorem \ref{thm1} are uninformative otherwise) and was made for the purpose of simplifying the calculations and to allow us to obtain a compact final bound (it allows the bound of Lemma \ref{sj3sk3} to be given in a compact form). This assumption is also natural, because we require that $np_*/r\rightarrow\infty$ in order for $T_{\lambda}\rightarrow_d \chi_{(r-1)}^2$ (see Remark \ref{remconds}). The assumption $np_*\geq 2(\lambda-3)^2$, $\lambda\geq3$, is also mild (our bound cannot be small unless this condition is met) and natural because a necessary condition for $T_{\lambda}\rightarrow_d \chi_{(r-1)}^2$ is that $np_*/\lambda^2\rightarrow\infty$ (see Remark \ref{remconds}). This assumption is needed to apply Lemma \ref{lem3.2} and, again, is useful in allowing a compact bound to be given. 
\end{remark}

\begin{remark}[The class of functions and IPMs] The bound (\ref{pdbound}) of Theorem \ref{thm1} is in a sense preferable to the bound (\ref{pdbound22}) because in the case $\lambda=1$ it has a better dependence on $r$.
 However, an advantage of the bound (\ref{pdbound22}) is that it holds for a wider class of test functions $h$ whilst having the same dependence on all parameters $n$, $\lambda$, $r$ and $p_1,\ldots,p_r$ if $\lambda\not=1$. We now see how this advantage plays out if the smooth test function bounds are converted into bounds involving integral probability metrics.

The bounds of Theorems \ref{thm0}--\ref{thm2} can be expressed in terms of convergence determining integral probability metrics (IPMs) \cite{gs02,z93} as follows. Let $\mathcal{H}_{q_1,q_2}=\{h\in C_b^{q_1,q_2}(\mathbb{R}^+)\,:\,\|h^{(q_1)}\|\leq1,\|h^{(q_1+1)}\|\leq1,\ldots,\|h^{(q_2-1)}\|\leq1,\|h^{(q_2)}\|\leq1\}$, where $h^{(0)}\equiv h$. Then we define the IPM between the laws of real-valued random variables $X$ and $Y$ by $d_{q_1,q_2}(\mathcal{L}(X),\mathcal{L}(Y))=\sup_{h\in\mathcal{H}_{q_1,q_2}}|\mathbb{E}[h(X)]-\mathbb{E}[h(Y)]|$. 
As an example, for $Y_{r-1}\sim\chi_{(r-1)}^2$, from (\ref{pdbound22}) we obtain the two-sided inequality
\begin{align}|\mathbb{E}[T_\lambda]-\mathbb{E}[Y_{r-1}]|\leq d_{1,5}(\mathcal{L}(T_\lambda),\chi_{(r-1)}^2)\leq \frac{1}{n}\bigg(\sum_{j=1}^r\frac{1}{p_j}\bigg)\bigg\{665\,476\sqrt{r}&\nonumber\\
\label{ipmineq}+101\,997|\lambda-1|r+\frac{19}{9}(\lambda-1)^2+\frac{|(\lambda-1)(\lambda-2)|(12\lambda+13)}{6(\lambda+1)}\bigg\},&
\end{align}
where the lower bound follows because $h(x)=x$ is in the class $\mathcal{H}_{1,5}$. If $r\geq2$, then $\mathbb{E}[T_1]=\mathbb{E}[\chi^2]=\mathbb{E}[Y_{r-1}]$, and it was shown by \cite[p.\ 451]{cr84} (here we have written their formula in a slightly different form) that, for $\lambda>-1$,
\begin{align}\label{momapprox}|\mathbb{E}[T_\lambda]-\mathbb{E}[Y_{r-1}]|&=\bigg|\sum_{j=1}^r\bigg\{\frac{(\lambda-1)}{np_j}\bigg[\frac{3\lambda-2}{12}-\frac{\lambda p_j}{2}+\frac{3\lambda+2}{12}p_j^2\bigg]\bigg\}+O(n^{-3/2})\bigg|.
\end{align}
It should be noted that $\mathbb{E}[T_{-1}]=\infty$, and that the $O(n^{-3/2})$ term in the asymptotic approximation for $|\mathbb{E}[T_\lambda]-\mathbb{E}[Y_{r-1}]|$ therefore blows up as $\lambda\downarrow-1$. We also note that comparing the lower and upper bounds in (\ref{ipmineq}) shows that the upper bound (\ref{ipmineq}) has a good dependence on all parameters;  $n$, $\lambda$, $r$ and $p_1,\ldots,p_r$. This is the subject of the next remark.
\end{remark}


\begin{remark}[The dependence of the bounds on the parameters and conditions for convergence]\label{remconds}
A basic necessary condition for $T_\lambda$ to converge in distribution to $\chi_{(r-1)}^2$  is that $\mathbb{E}[T_\lambda]\rightarrow\mathbb{E}[Y_{r-1}]$, for $Y_{r-1}\sim\chi_{(r-1)}^2$. Therefore, we can study the optimality of the bounds of Theorems \ref{thm1} and \ref{thm2} on the parameters $n$, $\lambda$, $r$, and $p_1,\ldots,p_r$ via comparison to the asymptotic approximation (\ref{momapprox}). We observe that the $O(n^{-1})$ rate of convergence  of the upper bounds in Theorem \ref{thm1} is optimal. Moreover, the dependence on $p_*$ of the form $(np_*)^{-1}$ is also best possible, as well as the the dependence on $\lambda$ of the form $\lambda^2/n$ for large $n$ and $\lambda$. 
The bounds in Theorem \ref{thm2} are of the sub-optimal order $O(n^{-1/2})$, but otherwise also have the correct dependence on $p_*$ and $\lambda$, because they tend to zero if and only if $np_*\rightarrow\infty$ (with fixed $r$ and $\lambda$) and $n/\lambda^2\rightarrow\infty$ (with fixed $r$ and $p_*$).

Let us now consider the dependence of the bounds of Theorems \ref{thm1} and \ref{thm2} on $r$. To do this, we consider the case of uniform cell classification probabilities, $p_1=\cdots=p_r=1/r$, and suppose $\lambda\not=1$ is fixed. Then for large $r$, $|\mathbb{E}[T_\lambda]-\mathbb{E}[Y_{r-1}]|\rightarrow0$ if and only if $n/r^2\rightarrow\infty$, except when $\lambda=2/3$, in which case $|\mathbb{E}[T_{2/3}]-\mathbb{E}[Y_{r-1}]|\rightarrow0$ if and only if $n/r\rightarrow\infty$. In this case of uniform cell classification probabilities, the bounds of Theorems \ref{thm1} and \ref{thm2} converge to zero, for fixed $\lambda\not=1$, if and only if $n/r^3\rightarrow\infty$. We consider it most likely that $T_\lambda\rightarrow_d\chi_{(r-1)}^2$ if $n/r^2\rightarrow\infty$ and that our bounds therefore do not have an optimal dependence on $r$. In the proof of Theorem \ref{thm1}, our bounds for all the remainder terms except $R_3$ decay to zero if and only if $n/r^2\rightarrow\infty$, whilst our bound for $R_3$ decays to zero if and only if $n/r^3\rightarrow\infty$. The bound for $R_3$ follows from a long calculation involving analysis of the multivariate normal Stein equation; see Part II of the proof of Theorem 4.2 of \cite{gaunt chi square}. It is plausible that a more refined analysis could lead to a bound for $R_3$ with an improved dependence on $r$. We consider this to be an interesting but challenging problem.

If we consider the possibility that as $n\rightarrow\infty$ that $p_*$ may decay to zero and that $\lambda\not=1$ and $r$ may tend to infinity, we can read off from the upper bounds of Theorem \ref{thm1} that $T_\lambda\rightarrow_d\chi_{(r-1)}^2$ if $n/(S\lambda\max\{r,\lambda\})\rightarrow\infty$, where $S=\sum_{j=1}^rp_j^{-1}$. If we suppose that $r/\lambda\rightarrow c$, where $0\leq c<\infty$, then we see that $T_\lambda\rightarrow_d\chi_{(r-1)}^2$ if $n/(S\lambda^2)\rightarrow\infty$. From (\ref{momapprox}) this can be seen to be an optimal condition (except perhaps if $\lambda=2/3$).
We conjecture that in fact $T_\lambda\rightarrow_d\chi_{(r-1)}^2$ if $n/(S\lambda^2)\rightarrow\infty$, even if $r\gg\lambda$ (except perhaps if $\lambda=2/3$). 

In stating the above conditions, we have been careful to give our conditions in terms of the sum $S$ rather than $p_*$, because in the $r\rightarrow\infty$ regime it is possible that $S\ll r/p_*$. For example, take $p_1=n^{-1}$, and $p_j=(n-1)/(n(r-1))$, $j=2,\ldots,r$, so that $\sum_{j=1}^rp_j=1$. Then $S=n+(r-1)^2n/(n-1)\ll rn=r/p_*$, if $r\ll n$. If, however, $r/(p_*S)\rightarrow C$, where $1\leq C<\infty$ (which will be the case unless there are some exceptionally small cell classification probabilities), then we can write the conditions in a simple and intuitive form by replacing $S$ by $r/p_*$. For example, our conjectured optimal condition for which $T_\lambda\rightarrow_d\chi_{(r-1)}^2$ would read $np_*/(r\lambda^2)\rightarrow\infty$.

From (\ref{pearbound01}), we see that even if $r\rightarrow\infty$, then $T_1=\chi^2\rightarrow_d\chi_{(r-1)}^2$ provided $np_*\rightarrow\infty$, which is a well-established condition for chi-square approximation of Pearson's statistic to be valid (see \cite{gn96}). Thus, in the $r\rightarrow\infty$ regime, it can be seen that Pearson's statistic converges in distribution faster than any other member of the power divergence family, except perhaps $T_{2/3}$ for which we are unable to provide a definitive answer in this paper. 
\end{remark}

\begin{remark}[Comparison to results from simulation studies]\label{remsimul} Simulation studies assessing the quality of chi-square approximation of the statistic $T_\lambda$ can be found in, amongst others, \cite{cr84,m91,read84a,rc88,rudas86}.  In this remark, we show how some of the findings of these studies are reflected in our bounds.
As is common with small-sample studies for goodness-of-fit tests, most of the studies assumed that the classification probabilities are uniform ($p_j=1/r$ for all $j=1,\ldots,r$) \cite[p.\ 339]{kl80}. Under this assumption, \cite{read84a} concluded that, for $\lambda\in[1/3,3/2]$, the chi-square approximation of $T_\lambda$ is accurate for the purpose of hypothesis testing at significance levels between 0.1 and 0.01, provided $r\leq6$ and $n\geq10$ [that the approximation can be accurate for sample sizes as small as $n=10$ is consistent with the fast $O(n^{-1})$ rate of convergence of the bounds of Theorem \ref{thm1}]. However, if $|\lambda|$ increases, the approximation becomes worse, and this is magnified as $r$ increases for fixed $n$ [our bounds grow as $\lambda$ and $r$ increase]. Outside the equiprobable hypothesis, it is harder to prescribe general rules \cite[p.\ 341]{kl80}, although the accuracy of the approximation diminishes \cite{ptw14} [as we move away from the equiprobable hypothesis the sum $\sum_{j=1}^rp_j^{-1}$ increases]. 

Through a combination of simulation studies and theoretical considerations, \cite{cr84} proposed the statistic $T_{2/3}$ as an alternative to the classical likelihood ratio and Pearson statistics. One of the reasons was because they observed that for the quantities $\mathbb{E}[T_{\lambda}^k]-\mathbb{E}[Y_{r-1}^k]$, $k=1,2,3$, for $Y_{r-1}\sim \chi_{(r-1)}^2$, the second order correction terms vanish in the limit $r\rightarrow\infty$ if and only if $\lambda=2/3,1$. Our bounds do convergence faster in the case $r\rightarrow\infty$ if $\lambda=1$, but not if $\lambda=2/3$, and we know that for $\lambda\in(-1,\infty)\setminus\{2/3,1\}$ the true convergence rate is indeed slower in this regime. It would be interesting to provide a definitive answer as to whether the convergence rate does speed up in the large $r$ regime if $\lambda=2/3$. We do, however, note that in our method of proof there seems to be nothing special about the case $\lambda=2/3$, so a  different approach would be required to provide a positive answer.

\end{remark}

\section{Preliminary lemmas}\label{sec3}


We begin by presenting a series of moment bounds
 for binomial random variables.

\begin{lemma}Let $S=(U-np)/\sqrt{np}$, where $U\sim \mathrm{Bin}(n,p)$, for $n\geq1$, $0<p<1$ such that $np\geq2$. Then
\begin{align}\label{3mombound}\mathbb{E}[|S|^3]&\leq 3, \\
\label{4mombound}\mathbb{E}[S^4]&\leq 4, \\
\label{6mombound}\mathbb{E}[S^6]&\leq 28.
\end{align}
\end{lemma}

\begin{proof} Using standard formulas for the fourth and sixth central moments of the binomial distribution (see, for example \cite{stuart}), we have that
\begin{align*}\mathbb{E}[S^4]&=\frac{1}{(np)^2}\cdot np(1-p)\big[3(2-n)p^2+3(n-2)p+1]\leq3+\frac{1}{np}\leq\frac{7}{2}<4,
\end{align*}
and
\begin{align*}\mathbb{E}[S^6]&=\frac{1}{(np)^3}\cdot np(1-p)\big[5(3n^2-26n+24)p^4-10(3n^2-26n+24)p^3\\
&\quad+5(3n^2-31n+30)p^2+5(5n-6)p+1\big]\leq 15+\frac{25}{np}+\frac{1}{(np)^2}<28,
\end{align*}
where we used that $n\geq1$, and $0<p<1$ and $np\geq2$ to simplify the bounds. Inequality (\ref{3mombound}) follows from (\ref{4mombound}) by H\"older's inequality: $\mathbb{E}[|S|^3]\leq(\mathbb{E}[S^4])^{3/4}\leq 4^{3/4}<3$.
\end{proof}



\begin{lemma}[T.\ D.\ Ahle \cite{a17}]\label{lem3.2}Let $U\sim \mathrm{Bin}(n,p)$, where $n\geq1$ and $0<p<1$. Let $k>0$. Then
\begin{equation}\label{bin000}\mathbb{E}[U^k] \leq \exp(k^2/(2np))\cdot(np)^k.
\end{equation}
Suppose now that $np\geq \alpha k^2$, where $\alpha>0$. Then it is immediate from (\ref{bin000}) that
\begin{equation}\label{binbound1}\mathbb{E}[U^k] \leq\mathrm{e}^{1/(2\alpha)}(np)^k.
\end{equation}
\end{lemma}


\begin{remark}Bounds of the form $\mathbb{E}[U^k]\leq C_k(np)^k$, where $C_k$ is an explicit constant only depending on $k$, in which no further restrictions on $n$ and $p$ are available (see \cite{p95}). However, for our purpose of obtaining bounds with the best possible dependence on all parameters, such a bound is not suitable, as it would lead to a worse dependence on $\lambda$ in the final bound. The additional assumption $np\geq \alpha k^2$ for inequality (\ref{binbound1}) leads to the assumption $np\geq 2(\lambda-3)^2$, $\lambda\geq3$, in Theorem \ref{thm1} (and a similar assumption in Theorem \ref{thm2}), which is mild and preferable to a worse dependence of the bounds on $\lambda$.
\end{remark}

\begin{lemma}\label{sj3sk3}Let $(U_1,\ldots,U_r$) be the multinomial vector of observed counts, as in Theorem \ref{thm1}, where $r\geq2$. Suppose also that $np_j\geq r$ for all $j=1,\ldots,r$. For $j=1,\ldots,r$, let $S_j=(U_j-np_j)/\sqrt{np_j}$. Then, for $j\not= k$,
\begin{equation}\label{sj3sk3z}|\mathbb{E}[S_j^3S_k^3]|\leq\frac{6}{r}+4\sqrt{p_jp_k}.
\end{equation} 
\end{lemma}

\begin{proof} Let $I_j(i)$ be the indicator that trial $i$ results in classification in cell $j$, and let $\tilde{I_j}(i) = I_j(i) - p_j$ be its standardised version.  We can then write $S_j = (np_j)^{-1/2}\sum_{i=1}^n \tilde{I}_j(i)$. We note that $\tilde{I_{j_1}}(i_1)$ and $\tilde{I_{j_2}}(i_2)$ are independent for all $j_1,j_2\in\{1,\ldots,r\}$ if $i_1\not=i_2$. On using this property together with the fact that $\mathbb{E}[\tilde{I}_{j}(i)]=0$ for all $i=1,\ldots,n$ and $j=1,\ldots,r $, we obtain that, for $j\not= k$,
\begin{align}&\mathbb{E}[S_j^3S_k^3]=\frac{1}{n^3(p_jp_k)^{3/2}}\mathbb{E}\bigg[\bigg(\sum_{i=1}^n \tilde{I}_j(i)\bigg)^3\bigg(\sum_{\ell=1}^n\tilde{I}_k(\ell)\bigg)^3\bigg]\nonumber\\
&\quad=\frac{1}{n^3(p_jp_k)^{3/2}}\bigg\{\sum_{i_1=1}^n\mathbb{E}[\tilde{I}_j(i_1)^3\tilde{I}_k(i_1)^3]+\sum_{i_1=1}^n\sum_{i_2\not=i_1}^n\Big\{\mathbb{E}[\tilde{I}_j(i_1)^3]\mathbb{E}[\tilde{I}_k(i_2)^3]\nonumber\\
&\quad\quad+9\mathbb{E}[\tilde{I}_j(i_1)^2\tilde{I}_k(i_1)^2]\mathbb{E}[\tilde{I}_{j}(i_2)\tilde{I}_{k}(i_2)]+3\mathbb{E}[\tilde{I}_j(i_1)^3\tilde{I}_k(i_1)]\mathbb{E}[\tilde{I}_{k}(i_2)^2]\Big\}\nonumber\\
&\quad\quad+3\mathbb{E}[\tilde{I}_k(i_1)^3\tilde{I}_j(i_1)]\mathbb{E}[\tilde{I}_{j}(i_2)^2]+\sum_{i_1=1}^n\sum_{i_2\not=i_1}^n\sum_{i_3\not=i_1,i_2}^n\Big\{9\mathbb{E}[\tilde{I}_j(i_1)^2]\mathbb{E}[\tilde{I}_j(i_2)\tilde{I}_{k}(i_2)]\mathbb{E}[\tilde{I}_{k}(i_3)^2]\nonumber\\
\label{longcal}&\quad\quad+6\mathbb{E}[\tilde{I}_j(i_1)\tilde{I}_k(i_1)]\mathbb{E}[\tilde{I}_j(i_2)\tilde{I}_k(i_2)]\mathbb{E}[\tilde{I}_j(i_3)\tilde{I}_k(i_3)]\Big\}\bigg\}.
\end{align}
We now calculate the expectations in the above expression. We have that $\mathbb{E}[\tilde{I}_{j}(i_2)^2]=\mathrm{Var}(I_j(i_2))=p_j(1-p_j)$. Also, by a routine calculation (or using a standard formula for the third central moment of the Bernoulli distribution), we have that $\mathbb{E}[\tilde{I}_j(i_1)^3]=p_j(1-p_j)(1-2p_j)$. For the more complex expectations involving products of powers of $\tilde{I}_j(i_1)$ and $\tilde{I}_k(i_1)$, $j\not=k$, we use the fact that each trial leads to a unique classification. For $u,v\in\mathbb{N}$, 
\begin{align*}\mathbb{E}[\tilde{I}_j(i_1)^u\tilde{I}_k(i_1)^v]&=\mathbb{E}[(I_j(i_1)-p_j)^u(I_k(i_1)-p_k)^v]\\
&=-p_j^u(1-p_k)^v\mathbb{P}(I_k(i_1)=1)-p_k^v(1-p_j)^u\mathbb{P}(I_j(i_1)=1)\\
&\quad+p_j^up_k^v\mathbb{P}(I_j(i_1)=I_k(i_1)=0) \\
&=-p_j^u(1-p_k)^vp_k-p_k^v(1-p_j)^up_j+p_j^up_k^v(1-p_j-p_k)\\
&=p_jp_k\big[-p_j^{u-1}(1-p_k)^v-p_k^{v-1}(1-p_j)^u+(1-p_j-p_k)p_j^{u-1}p_k^{v-2}\big].
\end{align*}
In the case $u=v=1$, we have the simplification $\mathbb{E}[\tilde{I}_j(i_1)\tilde{I}_k(i_1)]=-p_jp_k$. 
Substituting these formulas for the expectations present in (\ref{longcal}) now gives us
\begin{align*}\mathbb{E}[S_j^3S_k^3]&=\frac{1}{n^3(p_jp_k)^{3/2}}\Big[n\cdot p_jp_k\big[-p_j^2(1-p_k)^3-p_k^2(1-p_j)^3+(1-p_j-p_k)(p_jp_k)^2\big]\\
&\quad+n(n-1)\Big\{\Big(p_jp_k(1-p_j)(1-p_k)(1-2p_j)(1-2p_k)\Big)\\
&\quad+9\Big((p_jp_k)^2\big[p_k(1-p_j)^2+p_j(1-p_k)^2+p_j+p_k-1\big]\Big)\\
&\quad+3\Big(p_jp_k^2(1-p_k)(1-2p_k)\big[1-p_j-p_k-p_j^2(1-p_k)-(1-p_j)^3\big]\Big)\\
&\quad+3\Big(p_j^2p_k(1-p_j)(1-2p_j)\big[1-p_k-p_j-p_k^2(1-p_j)-(1-p_k)^3\big]\Big)\Big\}\\
&\quad+n(n-1)(n-2)\Big\{9\Big(-(p_jp_k)^2(1-p_j)(1-p_k)\Big)+6\Big(-(p_jp_k)^3\Big)\Big\}\Big].
\end{align*}
Elementary calculations that make use of the facts that $p_j,p_k\in(0,1)$ and that $0\leq p_j+p_k\leq 1$ now yield the bound
\begin{align}
|\mathbb{E}[S_j^3S_k^3]|&\leq \frac{1}{n^3(p_jp_k)^{3/2}}\Big\{np_jp_k+n^2p_jp_k+9n^2(p_jp_k)^2+3n^2p_jp_k^2+3n^2p_j^2p_k\nonumber\\
&\quad+9n^3(p_jp_k)^2+6n^3(p_jp_k)^3\Big\} \nonumber\\
&=\frac{1}{n^2\sqrt{p_jp_k}}+\frac{1}{4n\sqrt{p_jp_k}}+\frac{9\sqrt{p_jp_k}}{n}+\frac{3p_k^{3/2}}{n\sqrt{p_j}}+\frac{3p_j^{3/2}}{n\sqrt{p_k}}+\frac{9}{4}\sqrt{p_jp_k}+6(p_jp_k)^{3/2}\nonumber\\
&=\frac{1}{n\cdot n\sqrt{p_jp_k}}+\frac{1}{4n\sqrt{p_jp_k}}+\frac{9p_jp_k}{n\sqrt{p_jp_k}}+\frac{3\sqrt{p_j}p_k^{3/2}}{np_j}+\frac{3p_j^{3/2}\sqrt{p_k}}{np_k}+\frac{9}{4}\sqrt{p_jp_k}\nonumber\\
\label{longcal2}&\quad+6(p_jp_k)\cdot\sqrt{p_jp_k}.
\end{align}
In obtaining this bound, we did not attempt to optimise numerical constants, although in bounding the second and sixth terms we did take advantage of the simple inequality $(1-p_j)(1-p_k)\leq1/4$ (maximum at $p_j=p_k=1/2$). The final equality in (\ref{longcal2}) was given in anticipation of obtaining our final bound (\ref{sj3sk3z}) for $|\mathbb{E}[S_j^3S_k^3]|$. We now bound the terms in (\ref{longcal2}) using the following considerations. By assumption, $np_j\geq r$ for all $j=1,\ldots,r$, which also implies that $n\geq4$, as $r\geq2$. Moreover, for $j\not=k$, $p_jp_k\leq1/4$ (maximised for $p_j=p_k=1/2$) and $\sqrt{p_j}p_k^{3/2}\leq3\sqrt{3}/16$ (maximised at $p_j=1/4$, $p_k=3/4$), with the same bound holding for $p_j^{3/2}\sqrt{p_k}$ by symmetry. Making these considerations, we arrive at the bound
\begin{align*}|\mathbb{E}[S_j^3S_k^3]|\leq\frac{1}{4r}+\frac{1}{4r}+\frac{9}{4r}+2\cdot3\cdot\frac{3\sqrt{3}}{16r}+\frac{9}{4}\sqrt{p_jp_k}+\frac{6}{4}\sqrt{p_jp_k} <\frac{6}{r}+4\sqrt{p_jp_k},
\end{align*}
where in obtaining the second inequality we rounded the numerical constants up to the nearest integer. The proof of the lemma is complete.
\end{proof}

The following lemma, which is proved in Appendix \ref{appa}, will be needed in the proof of Theorem \ref{thm1}. Lemma \ref{lem2} below will be needed in the proof of Theorem \ref{thm2}. The proof of Lemma \ref{lem2} is similar to that of Lemma \ref{lem1} and is given in the Supplementary Information.

\begin{lemma}\label{lem1}Let $a>0$. 
 Then, 

\begin{itemize}

\item[(i)] For $x\geq0$,
\begin{equation}\label{lembound1}\bigg|2x\log\bigg(\frac{x}{a}\bigg)-2(x-a)-\frac{(x-a)^2}{a}+\frac{(x-a)^3}{3a^2}\bigg|\leq\frac{2(x-a)^4}{3a^3}.
\end{equation}

\item[(ii)] Suppose $\lambda\geq3$. Then, for $x\geq0$,
\begin{align}&\bigg|\frac{x^{\lambda+1}}{a^\lambda}-a-(\lambda+1)(x-a)-\frac{\lambda(\lambda+1)}{2a}(x-a)^2-\frac{(\lambda-1)\lambda(\lambda+1)}{6a^2}(x-a)^3\bigg|\nonumber\\
\label{lembound3}&\quad\leq \frac{(\lambda-2)(\lambda-1)\lambda(\lambda+1)}{24}\bigg(1+\Big(\frac{x}{a}\Big)^{\lambda-3}\bigg)\frac{(x-a)^4}{a^3}.
\end{align}

\item[(iii)] Suppose $\lambda\in(-1,3)\setminus\{0\}$. Then, for $x\geq0$,
\begin{align}&\bigg|\frac{x^{\lambda+1}}{a^\lambda}-a-(\lambda+1)(x-a)-\frac{\lambda(\lambda+1)}{2a}(x-a)^2-\frac{(\lambda-1)\lambda(\lambda+1)}{6a^2}(x-a)^3\bigg|\nonumber\\
\label{lembound32}&\quad\leq \frac{|(\lambda-2)(\lambda-1)\lambda|}{6}\frac{(x-a)^4}{a^3}.
\end{align}

\end{itemize}

\end{lemma}

\begin{lemma}\label{lem2} Let $a>0$. Then, 

\begin{itemize}

\item[(i)] For $x\geq0$,
\begin{equation}\label{lembound2}\bigg|2x\log\bigg(\frac{x}{a}\bigg)-2(x-a)-\frac{(x-a)^2}{a}\bigg|\leq\frac{|x-a|^3}{a^2}.
\end{equation}

\item[(ii)] Suppose $\lambda\geq2$. Then, for $x\geq0$,
\begin{align}&\bigg|\frac{x^{\lambda+1}}{a^\lambda}-a-(\lambda+1)(x-a)-\frac{\lambda(\lambda+1)}{2a}(x-a)^2\bigg|\nonumber\\
\label{lembound4}&\quad\leq \frac{(\lambda-1)\lambda(\lambda+1)}{6}\bigg(1+\Big(\frac{x}{a}\Big)^{\lambda-2}\bigg)\frac{|x-a|^3}{a^2}.
\end{align}

\item[(iii)] Suppose $\lambda\in(-1,2)\setminus\{0\}$. Then, for $x\geq0$,
\begin{align}\label{lembound42}&\bigg|\frac{x^{\lambda+1}}{a^\lambda}-a-(\lambda+1)(x-a)-\frac{\lambda(\lambda+1)}{2a}(x-a)^2\bigg|\leq \frac{|(\lambda-1)\lambda|}{2}\frac{|x-a|^3}{a^2}.
\end{align}

\end{itemize}

\end{lemma}

\section{Proofs of main results 
}\label{sec4}

In this section, we prove all the results stated in Section \ref{sec2}, except for Theorem \ref{thm0}, which is proved in Appendix \ref{appchi}. We begin by introducing some notation and make some basic observations regarding the power divergence statistic $T_\lambda$. For $j=1,\ldots,r$, let $S_j=(U_j-np_j)/\sqrt{np_j}$ denote the standardised cell counts, and let $\boldsymbol{S}=(S_1,\ldots,S_r)^\intercal$. Observe that $\sum_{j=1}^rU_j=n$ and $\sum_{j=1}^r\sqrt{p_j}S_j=0$, and that $U_j\sim\mathrm{Bin}(n,p_j)$, $j=1,\ldots,r$. With this notation, we may write the power divergence statistic $T_\lambda$, Pearson's statistic $\chi^2$, and the likelihood ratio statistic $L$ as
\begin{align*}T_\lambda&=\frac{2}{\lambda(\lambda+1)}\bigg[\sum_{j=1}^rnp_j\bigg(1+\frac{S_j}{\sqrt{np_j}}\bigg)^{\lambda+1}-n\bigg],\\
\chi^2&=\sum_{j=1}^r S_j^2, \\
L&=T_0=2\sum_{j=1}^r\big(np_j+\sqrt{np_j}S_j\big)\log\bigg(1+\frac{S_j}{\sqrt{np_j}}\bigg).
\end{align*}



\noindent{\emph{Proof of Theorem \ref{thm1}.}} For our purpose of deriving an order $n^{-1}$ bound, it will be convenient to write the power divergence statistic $T_\lambda$ in the form
\begin{equation*}T_\lambda=\sum_{j=1}^rS_j^2+\frac{\lambda-1}{3}\sum_{j=1}^r\frac{S_j^3}{\sqrt{np_j}}+R_\lambda(\boldsymbol{S}),
\end{equation*}
where, for $\lambda\in(-1,\infty)\setminus\{0\}$,
\begin{align*}R_\lambda(\boldsymbol{S})&=\frac{2}{\lambda(\lambda+1)}\sum_{j=1}^r\bigg[np_j\bigg(1+\frac{S_j}{\sqrt{np_j}}\bigg)^{\lambda+1}-np_j-(\lambda+1)\sqrt{np_j}S_j-\frac{\lambda(\lambda+1)}{2}S_j^2\\
&\quad-\frac{(\lambda-1)\lambda(\lambda+1)}{6\sqrt{np_j}}S_j^3\bigg],
\end{align*}
and, for $\lambda=0$,
\begin{align*}R_0(\boldsymbol{S})&=\sum_{j=1}^r\bigg[2\big(np_j+\sqrt{np_j}S_j\big)\log\bigg(1+\frac{S_j}{\sqrt{np_j}}\bigg)-2\sqrt{np_j}S_j-S_j^2-\frac{S_j^3}{3\sqrt{np_j}}\bigg].
\end{align*}
Here we used that $\sum_{j=1}^rp_j=1$ and $\sum_{j=1}^r\sqrt{p_j}S_j=0$.

Let us now bound the quantity of interest $|\mathbb{E}[h(T_\lambda)]-\chi_{(r-1)}^2h|$. By a first-order Taylor expansion and the triangle inequality we have that
\begin{align*}|\mathbb{E}[h(T_\lambda)]-\chi_{(r-1)}^2h|&=\bigg|\mathbb{E}\bigg[h\bigg(\sum_{j=1}^rS_j^2+\frac{\lambda-1}{3}\sum_{j=1}^r\frac{S_j^3}{\sqrt{np_j}}+R_\lambda(\boldsymbol{S})\bigg)\bigg]-\chi_{(r-1)}^2h\bigg|\\
&\leq\bigg|\mathbb{E}\bigg[h\bigg(\sum_{j=1}^rS_j^2+\frac{\lambda-1}{3}\sum_{j=1}^r\frac{S_j^3}{\sqrt{np_j}}\bigg)\bigg]-\chi_{(r-1)}^2h\bigg|+R_1,
\end{align*}
where
\begin{equation*}R_1=\|h'\|\mathbb{E}|R_\lambda(\boldsymbol{S})|.
\end{equation*}
Another Taylor expansion and application of the triangle inequality gives us the bound
\begin{equation*}|\mathbb{E}[h(T_\lambda)]-\chi_{(r-1)}^2h|\leq R_1+R_2+R_3+R_4,
\end{equation*}
where
\begin{eqnarray*}R_2&=&|\mathbb{E}[h(\chi^2)]-\chi_{(r-1)}^2h|,\\
R_3&=&\frac{|\lambda-1|}{3}\bigg|\mathbb{E}\bigg[\sum_{j=1}^r\frac{S_j^3}{\sqrt{np_j}}h'\bigg(\sum_{k=1}^rS_k^2\bigg)\bigg]\bigg|, \\
R_4&=&\frac{(\lambda-1)^2}{18}\|h''\|\mathbb{E}\bigg[\bigg(\sum_{j=1}^r\frac{S_j^3}{\sqrt{np_j}}\bigg)^2\bigg],
\end{eqnarray*}
and we recall that Pearson's statistic is given by $\chi^2=\sum_{j=1}^rS_j^2$.

We can bound the remainder $R_2$ immediately using either the bound (\ref{pearbound1}) or the bound (\ref{pearbound1bb}). To bound $R_3$, we recall another bound of \cite{gaunt chi square} (see the bound for $|\mathbb{E}h_2(\boldsymbol{S})|$ given on p.\ 747): for $j=1,\ldots,r$,
\begin{equation*}\bigg|\mathbb{E}\bigg[S_j^3h'\bigg(\sum_{k=1}^rS_k^2\bigg)\bigg]\bigg|\leq\big\{2\|h'\|+202\|h''\|+819\|h^{(3)}\|+100\,974\|h^{(4)}\|\big\}\sum_{k=1}^r\frac{1}{\sqrt{np_k}}.
\end{equation*}
Using this inequality we obtain the bound
\begin{equation*}R_3\leq\frac{|\lambda-1|}{3n}\bigg(\sum_{j=1}^r\frac{1}{\sqrt{p_j}}\bigg)^2\big\{2\|h'\|+202\|h''\|+819\|h^{(3)}\|+100\,974\|h^{(4)}\|\big\}.
\end{equation*}

Let us now bound $R_4$. Using the moment bounds (\ref{6mombound}) and (\ref{sj3sk3z}) (our assumption that $np_j\geq r$ for all $j=1,\ldots r$ allows us to apply this bound) gives
\begin{align*}\mathbb{E}\bigg[\bigg(\sum_{j=1}^r\frac{S_j^3}{\sqrt{np_j}}\bigg)^2\bigg]&=\sum_{j=1}^r\frac{\mathbb{E}[S_j^6]}{np_j}+\sum_{j=1}^r\sum_{k\not= j}^r\frac{\mathbb{E}[S_j^3S_k^3]}{n\sqrt{p_jp_k}} \\
&\leq\sum_{j=1}^r\frac{28}{np_j}+\sum_{j=1}^r\sum_{k\not= j}\bigg\{\frac{6}{nr\sqrt{p_jp_k}}+\frac{4}{n}\bigg\}\leq\frac{38}{nr}\bigg(\sum_{j=1}^r\frac{1}{\sqrt{p_j}}\bigg)^2.
\end{align*}
Therefore
\begin{equation*}R_4\leq\frac{19(\lambda-1)^2}{9nr}\|h''\|\bigg(\sum_{j=1}^r\frac{1}{\sqrt{p_j}}\bigg)^2.
\end{equation*}

Lastly, we bound $R_1$. We consider the cases $\lambda=0$ (the likelihood ratio statistic), $\lambda\geq3$ and $\lambda\in(-1,3)\setminus\{0\}$ separately.  Using that $S_j=(U_j-np_j)/\sqrt{np_j}$ and then applying inequality (\ref{lembound1}) of Lemma \ref{lem1}, we have that, for $\lambda=0$,
\begin{align}R_1&=\|h'\|\mathbb{E}|R_0(\boldsymbol{S})|\nonumber\\
&=\|h'\|\mathbb{E}\bigg|\sum_{j=1}^r\bigg[2U_j\log\bigg(\frac{U_j}{np_j}\bigg)-2(U_j-np_j)-\frac{(U_j-np_j)^2}{np_j}+\frac{(U_j-np_j)^3}{3(np_j)^2}\bigg]\bigg|\nonumber\\
\label{r1first}&\leq\frac{2\|h'\|}{3}\sum_{j=1}^r\frac{\mathbb{E}[(U_j-np_j)^4]}{(np_j)^3}\leq\frac{8\|h'\|}{3}\sum_{j=1}^r\frac{1}{np_j},
\end{align}
where we used that $\mathbb{E}[(U_j-np_j)^4]\leq4(np_j)^2$ in the final step (see (\ref{4mombound})).

Suppose now that $\lambda\geq3$. Applying inequality (\ref{lembound3}) of Lemma \ref{lem1}, we have, for $\lambda\geq3$,
\begin{align}R_1&=\|h'\|\mathbb{E}|R_\lambda(\boldsymbol{S})|\nonumber\\
&=\frac{2\|h'\|}{\lambda(\lambda+1)}\mathbb{E}\bigg|\sum_{j=1}^r\bigg[\frac{U_j^{\lambda+1}}{(np_j)^\lambda}-np_j-(\lambda+1)(U_j-np_j)-\frac{\lambda(\lambda+1)}{2}\frac{(U_j-np_j)^2}{np_j}\nonumber\\
\label{ghjkk}&\quad-\frac{(\lambda-1)\lambda(\lambda+1)}{6}\frac{(U_j-np_j)^3}{(np_j)^2}\bigg]\bigg|\\
&\leq\frac{(\lambda-2)(\lambda-1)\|h'\|}{12}\sum_{j=1}^r\mathbb{E}\bigg[\bigg(1+\frac{U_j^{\lambda-3}}{(np_j)^{\lambda-3}}\bigg)\frac{(U_j-np_j)^4}{(np_j)^3}\bigg]\nonumber\\
&\leq\frac{(\lambda-2)(\lambda-1)\|h'\|}{12}\sum_{j=1}^r\bigg\{\frac{\mathbb{E}[(U_j-np_j)^4]}{(np_j)^3}\nonumber\\
&\quad+\frac{1}{(np_j)^{\lambda}}\big(\mathbb{E}\big[(U_j-np_j)^6\big]\big)^{2/3}\big(\mathbb{E}[U_j^{3\lambda-9}]\big)^{1/3}\bigg\}\nonumber\\
&\leq\frac{(\lambda-2)(\lambda-1)\|h'\|}{12}\sum_{j=1}^r\bigg\{\frac{4}{np_j}+\frac{28^{2/3}\times (\mathrm{e}^{9/4})^{1/3}}{np_j}\bigg\}\nonumber\\
\label{r1second}&\leq 2(\lambda-2)(\lambda-1)\|h'\|\sum_{j=1}^r\frac{1}{np_j}.
\end{align}
Here we used H\"older's inequality to obtain the second inequality; to get the third inequality we used the moment bounds (\ref{4mombound}), (\ref{6mombound}) and (\ref{binbound1}), respectively, $\mathbb{E}\big[(U_j-np_j)^4\big]\leq4(np_j)^2$, $\mathbb{E}\big[(U_j-np_j)^6\big]\leq28(np_j)^3$ and $\mathbb{E}[U_j^{3\lambda-9}]\leq \mathrm{e}^{9/4}(np_j)^{3\lambda-9}$, where the latter bound is valid due to our assumption that $np_*\geq2(\lambda-3)^2$ for $\lambda\geq3$.  

Lastly, we consider the case $\lambda\in(-1,3)\setminus\{0\}$. Applying inequality (\ref{lembound32}) of Lemma \ref{lem1} to (\ref{ghjkk}) and proceeding similarly to before gives that, for $\lambda\in(-1,3)\setminus\{0\}$,
\begin{align}\label{r1third}R_1\leq\frac{|(\lambda-2)(\lambda-1)|\|h'\|}{6(\lambda+1)}\sum_{j=1}^r\frac{\mathbb{E}[(U_j-np_j)^4]}{(np_j)^3}\leq \frac{2|(\lambda-2)(\lambda-1)|\|h'\|}{3(\lambda+1)}\sum_{j=1}^r\frac{1}{np_j}.
\end{align}
To obtain a universal bound on $|\mathbb{E}[h(T_\lambda)]-\chi_{(r-1)}^2h|$ that is valid for all $\lambda>-1$, we observe that we can take the following upper bound for $R_1$ (valid for $\lambda\geq3$ under the condition that $np_*\geq 2(\lambda-3)^2$)
\begin{align}\label{r1final}R_1\leq\frac{|(\lambda-1)(\lambda-2)(12\lambda+13)|\|h'\|}{6(\lambda+1)}\sum_{j=1}^r\frac{1}{np_j},
\end{align}
which can be seen to be greater than each of the upper bounds (\ref{r1first}), (\ref{r1second}) and (\ref{r1third}) that were obtained for the cases $\lambda=0$, $\lambda\geq3$ and $\lambda\in(-1,3)\setminus\{0\}$, respectively. 

Finally, we sum up our bound (\ref{r1final}) for $R_1$ and our bounds for $R_2$, $R_3$ and $R_4$. To obtain a compact final bound we use the inequality $\big(\sum_{j=1}^rp_j^{-1/2}\big)^2\leq r\sum_{j=1}^rp_j^{-1}$, which follows from the Cauchy-Schwarz inequality. Using (\ref{pearbound1}) to bound $R_2$ gives us the bound (\ref{pdbound}), whilst using (\ref{pearbound1bb}) to bound $R_2$ yields (\ref{pdbound22}). This completes the proof. \hfill $\Box$

\begin{remark}It was a fortunate accident that the term $R_3$ had previously been bounded in \cite{gaunt chi square}. That this term is $O(n^{-1})$ is in some sense the key reason as to why the $O(n^{-1})$ rate is attained. Indeed, through our use of Taylor expansions we could guarantee that the remainder terms $R_1$ and $R_4$ would be $O(n^{-1})$, and $R_2=|\mathbb{E}[h(\chi^2)]-\chi_{(r-1)}^2h|$ had of course already been shown to be $O(n^{-1})$ by \cite{gaunt chi square}. To give a self-contained account as to why the $O(n^{-1})$ rate is attained, we therefore briefly sketch why $R_3=O(n^{-1})$.

Let $g(\boldsymbol{s})=\sum_{j=1}^r p_j^{-1/2}s_j^3h'(\sum_{k=1}^r s_k^2)$. Observe that $g(-\boldsymbol{s})=-g(\boldsymbol{s})$. Now, let $\boldsymbol{Z}$ be a centered $r$-dimensional multivariate normal random vector with covariance matrix equal to that of $\boldsymbol{S}$. Then, because $\boldsymbol{Z}=_d-\boldsymbol{Z}$, we have that $\mathbb{E}[g(\boldsymbol{Z})]=-\mathbb{E}[g(\boldsymbol{-Z})]=-\mathbb{E}[g(\boldsymbol{Z})]$, so that $\mathbb{E}[g(\boldsymbol{Z})]=0$. Therefore
\begin{align*}R_3=\frac{|\lambda-1|}{3\sqrt{n}}|\mathbb{E}[g(\boldsymbol{S})]|=\frac{|\lambda-1|}{3\sqrt{n}}|\mathbb{E}[g(\boldsymbol{S})]-\mathbb{E}[g(\boldsymbol{Z})]|.
\end{align*}
The quantity $|\mathbb{E}[g(\boldsymbol{S})]-\mathbb{E}[g(\boldsymbol{Z})]|$ can then be bounded to order $O(n^{-1/2})$ by Stein's method for multivariate normal approximation with polynomial growth rate test functions \cite{gaunt normal}.
\end{remark}



\noindent{\emph{Proof of Theorem \ref{thm2}.}} The proof is similar to that of Theorem \ref{thm1}, but a little shorter and simpler. For our aim of obtaining an order $n^{-1/2}$ bound, we express $T_\lambda$ in the form 
\begin{equation*}T_\lambda=\sum_{j=1}^rS_j^2+R_\lambda(\boldsymbol{S})=\chi^2+R_\lambda(\boldsymbol{S}),
\end{equation*}
where, for $\lambda\in(-1,\infty)\setminus\{0\}$,
\begin{align*}R_\lambda(\boldsymbol{S})&=\frac{2}{\lambda(\lambda+1)}\sum_{j=1}^r\bigg[np_j\bigg(1+\frac{S_j}{\sqrt{np_j}}\bigg)^{\lambda+1}-np_j-(\lambda+1)\sqrt{np_j}S_j-\frac{\lambda(\lambda+1)}{2}S_j^2\bigg],
\end{align*}
and, for $\lambda=0$,
\begin{align*}R_0(\boldsymbol{S})&=\sum_{j=1}^r\bigg[2\big(np_j+\sqrt{np_j}S_j\big)\log\bigg(1+\frac{S_j}{\sqrt{np_j}}\bigg)-2\sqrt{np_j}S_j-S_j^2\bigg].
\end{align*}
Now, by a first-order Taylor expansion and the triangle inequality we have that
\begin{align*}|\mathbb{E}[h(T_\lambda)]-\chi_{(r-1)}^2h|=|\mathbb{E}[h(\chi^2+R_\lambda(\boldsymbol{S}))]-\chi_{(r-1)}^2h|\leq|\mathbb{E}[h(\chi^2)]-\chi_{(r-1)}^2h|+R_1,
\end{align*}
where
\begin{equation*}R_1=\|h'\|\mathbb{E}|R_\lambda(\boldsymbol{S})|.
\end{equation*}

The quantity $|\mathbb{E}[h(\chi^2)]-\chi_{(r-1)}^2h|$ can be immediately bounded using either the bound (\ref{pearbound01}) or the bound (\ref{pearbound01bb}).
To bound $R_1$, we consider the cases $\lambda=0$, $\lambda\geq2$ and $\lambda\in(-1,2)\setminus\{0\}$ separately. For $\lambda=0$, using inequality (\ref{lembound2}) of Lemma \ref{lem2} we obtain the bound
\begin{align}R_1
&=\|h'\|\mathbb{E}\bigg|\sum_{j=1}^r\bigg[2U_j\log\bigg(\frac{U_j}{np_j}\bigg)-2(U_j-np_j)-\frac{(U_j-np_j)^2}{np_j}\bigg]\bigg|\nonumber\\
\label{r1firstvv}&\leq\|h'\|\sum_{j=1}^r\frac{\mathbb{E}[|U_j-np_j|^3]}{(np_j)^2}\leq 3\|h'\|\sum_{j=1}^r\frac{1}{\sqrt{np_j}},
\end{align}
where we used inequality (\ref{3mombound}) in the final step.

For $\lambda\geq2$ we use inequality (\ref{lembound4}) of Lemma \ref{lem2} to get
\begin{align}R_1&=\frac{2\|h'\|}{\lambda(\lambda+1)}\mathbb{E}\bigg|\sum_{j=1}^r\bigg[\frac{U_j^{\lambda+1}}{(np_j)^\lambda}-np_j-(\lambda+1)(U_j-np_j)-\frac{\lambda(\lambda+1)}{2}\frac{(U_j-np_j)^2}{np_j}\bigg]\bigg|\nonumber\\
&\leq\frac{(\lambda-1)\|h'\|}{3}\sum_{j=1}^r\mathbb{E}\bigg[\bigg(1+\frac{U_j^{\lambda-2}}{(np_j)^{\lambda-2}}\bigg)\frac{|U_j-np_j|^3}{(np_j)^2}\bigg]\nonumber\\
&\leq\frac{(\lambda-1)\|h'\|}{3}\sum_{j=1}^r\bigg\{\frac{\mathbb{E}[|U_j-np_j|^3]}{(np_j)^2}+\frac{1}{(np_j)^{\lambda}}\big(\mathbb{E}\big[(U_j-np_j)^6\big]\big)^{1/2}\big(\mathbb{E}[U_j^{2\lambda-4}]\big)^{1/2}\bigg\}\nonumber\\
\label{r1secondvv}&\leq\frac{(\lambda-1)\|h'\|}{3}\sum_{j=1}^r\bigg\{\frac{3}{\sqrt{np_j}}+\frac{28^{1/2}\times \mathrm{e}^{1/2}}{\sqrt{np_j}}\bigg\}\leq 4(\lambda-1)\|h'\|\sum_{j=1}^r\frac{1}{\sqrt{np_j}}.
\end{align}
Here we used the Cauchy-Schwarz inequality to obtain the second inequality; to get the third inequality we used the moment bounds (\ref{3mombound}), (\ref{6mombound}) and (\ref{binbound1}), respectively, $\mathbb{E}\big[(U_j-np_j)^4\big]\leq4(np_j)^2$, $\mathbb{E}\big[(U_j-np_j)^6\big]\leq28(np_j)^3$ and $\mathbb{E}[U_j^{2\lambda-4}]\leq \mathrm{e}(np_j)^{2\lambda-4}$, where the latter bound is valid due to our assumption that $np_*\geq2(\lambda-2)^2$ for $\lambda\geq2$.  

Lastly, we consider the case $\lambda\in(-1,2)\setminus\{0\}$. Using inequality (\ref{lembound42}) of Lemma \ref{lem2} to (\ref{ghjkk}) and proceeding similarly to before gives that, for $\lambda\in(-1,2)\setminus\{0\}$,
\begin{align}\label{r1thirdvv}R_1\leq\frac{|\lambda-1|\|h'\|}{\lambda+1}\sum_{j=1}^r\frac{\mathbb{E}[|U_j-np_j|^3]}{(np_j)^2}\leq \frac{3|\lambda-1|\|h'\|}{\lambda+1}\sum_{j=1}^r\frac{1}{\sqrt{np_j}}.
\end{align}
To obtain a universal bound on $|\mathbb{E}[h(T_\lambda)]-\chi_{(r-1)}^2h|$ that is valid for all $\lambda>-1$, we observe that we can take the following upper bound for $R_1$ (valid for $\lambda\geq2$ provided $np_*\geq 2(\lambda-2)^2$)
\begin{equation}\label{r1finalvv}R_1\leq\frac{|(\lambda-1)(4\lambda+7)|\|h'\|}{\lambda+1}\sum_{j=1}^r\frac{1}{\sqrt{np_j}},
\end{equation}
which can be seen to be greater than each of the upper bounds (\ref{r1firstvv}), (\ref{r1secondvv}) and (\ref{r1thirdvv}) that were obtained for the cases $\lambda=0$, $\lambda\geq2$ and $\lambda\in(-1,2)\setminus\{0\}$, respectively. 

Summing up our bound (\ref{r1finalvv}) for $R_1$ and either the bound (\ref{pearbound01}) or the bound (\ref{pearbound01bb}) for $|\mathbb{E}[h(\chi^2)]-\chi_{(r-1)}^2h|$ now yields the desired bounds (\ref{pdbound0}) and (\ref{pdbound022}), respectively. \hfill $\Box$

\vspace{3mm}

\noindent{\emph{Proof of Corollary \ref{cor1}.}}
Let $\alpha>0$, and for fixed $z>0$ define
\begin{equation*}h_\alpha(x)=\begin{cases} 1, & \: \mbox{if } x\leq z, \\
1-2(x-z)^2/\alpha^2, & \:  \mbox {if } z<x\leq z+\alpha/2, \\
2(x-(z+\alpha))^2/\alpha^2, & \:  \mbox {if } z+\alpha/2<x\leq z+\alpha, \\
0, & \:  \mbox {if } x\geq z+\alpha. \end{cases}
\end{equation*}
Then $h_\alpha'$ is Lipshitz with $\|h_\alpha\|=1$, $\|h_\alpha'\|=2/\alpha$ and $\|h_\alpha''\|=4/\alpha^2$.  Let $Y_{r-1}\sim\chi_{(r-1)}^2$. Then, using  (\ref{pdbound0}), together with the basic inequality $\sum_{j=1}^rp_j^{-1/2}\leq r/\sqrt{p_*}$, gives
\begin{align}&\mathbb{P}(T_\lambda\leq z)-\mathbb{P}(Y_{r-1}\leq z)\nonumber\\
&\leq \mathbb{E}[h_\alpha(T_\lambda)]-\mathbb{E}[h_\alpha(Y_{r-1})]+\mathbb{E}[h_\alpha(Y_{r-1})]-\mathbb{P}(Y_{r-1}\leq z)\nonumber \\
&\leq \frac{1}{\sqrt{np_*}}\bigg\{72\|h_\alpha\|+\bigg(552+\frac{|(\lambda-1)(4\lambda+7)|r}{\lambda+1}\bigg)\|h_\alpha'\|+1008\|h_\alpha''\|\bigg\}\nonumber\\
&\quad+\mathbb{P}(z\leq Y_{r-1}\leq z+\alpha)\nonumber \\
\label{kol123}&=\frac{1}{\sqrt{np_*}}\bigg\{72+\frac{1104}{\alpha}+\frac{|(\lambda-1)(4\lambda+7)|r}{(\lambda+1)\alpha}+\frac{4032}{\alpha^2}\bigg\}+\mathbb{P}(z\leq Y_{r-1}\leq z+\alpha).
\end{align}
It was shown by \cite[p.\ 754]{gaunt chi square} that
\begin{equation}\label{ccbbcsb}\mathbb{P}(z\leq Y_{r-1}\leq z+\alpha)\leq \begin{cases} \displaystyle \sqrt{2\alpha/\pi}, & \: \mbox{if } r=2, \\
\alpha/2, & \: \mbox{if } r=3, \\
\displaystyle \frac{\alpha}{2\sqrt{\pi(r-3)}}, & \:  \mbox {if } r\geq4. \end{cases}
\end{equation}
Upper bounds for the cases $r=2$, $r=3$ and $r\geq 4$ follow on substituting inequality (\ref{ccbbcsb}) into (\ref{kol123}) and selecting a suitable $\alpha$.  We choose $\alpha=52.75(np_*)^{-1/5}$ for $r=2$; we take $\alpha=25.27(np_*)^{-1/6}$ when $r=3$; and $\alpha=30.58(r-3)^{1/6}(np_*)^{-1/6}$ for $r\geq4$.  A lower bound can be obtained similarly, which is the negative of the upper bound.  The proof is now complete. \hfill $\square$

\appendix

\section{Stein's method for chi-square approximation and proof of Theorem \ref{thm0}}\label{appchi}

The first detailed study of Stein's method for chi-square approximation was given in the thesis \cite{luk}. For further details on Stein's method for chi-square approximation we refer the reader to \cite{dp18,gaunt chi square,nourdin1}. At the heart of Stein's method for chi-square approximation is the following so-called \emph{Stein equation} for the $\chi_{(p)}^2$ distribution (see \cite{dz91})
\begin{equation}\label{chisteineqn}xf''(x)+\frac{1}{2}(p-x)f'(x)=h(x)-\chi_{(p)}^2h,
\end{equation}
where $h$ is a real-valued test function and we recall that $\chi_{(d)}^2h$ denotes the quantity $\mathbb{E}[h(Y_{p})]$ for $Y_p\sim\chi_{(p)}^2$. It is straightforward to verify that 
\begin{equation}\label{chisteinsoln}f'(x)=\frac{1}{x\rho(x)}\int_0^x\big(h(t)-\chi_{(p)}^2h\big)\rho(t)\,\mathrm{d}t,
\end{equation}
where $\rho(x)$ denotes the $\chi_{(p)}^2$ density, solves the Stein equation (\ref{chisteineqn}) (see p.\ 59, Lemma 1 of \cite{stein2}). 

Suppose one wishes to obtain error bounds for the approximation of the distribution of a random variable of interest $W$ and a $\chi_{(p)}^2$ distribution. The Stein equation (\ref{chisteineqn}) allows one to obtain bounds on the quantity of interest $|\mathbb{E}[h(W)]-\chi_{(p)}^2h|$ by the following transfer principle: evaluate both sides of (\ref{chisteineqn}) at $W$ and take expectations to obtain
\begin{equation}\label{ewewew}|\mathbb{E}[h(W)]-\chi_{(p)}^2h|=\mathbb{E}\bigg[Wf''(W)+\frac{1}{2}(p-W)f'(W)\bigg],
\end{equation}
where $f$ is the solution (\ref{chisteinsoln}). Thus, the chi-square approximation problem is reduced to bounding the right-hand side of (\ref{ewewew}). For this procedure to be effective, one requires suitable estimates for the solution of the Stein equation together with certain lower order derivatives of the solution. We now record some bounds for solution (\ref{chisteinsoln}) that will be needed in the proof of Theorem \ref{thm0}. 

We first state a bound of \cite{luk}. For $h\in C_b^{k,k}(\mathbb{R}^+)$,
\begin{equation}\label{lukluk2}\|f^{(k)}\|\leq\frac{2\|h^{(k)}\|}{k}, \quad k\geq 1.
\end{equation}
The following bound of \cite{gaunt thesis} improved a bound of  \cite{pickett thesis}.  For $h\in C_b^{k-1,k-1}(\mathbb{R}^+)$,
\begin{equation}\label{gamma22}\|f^{(k)}\|\leq \bigg\{\frac{\sqrt{2}(\sqrt{2\pi}+\mathrm{e}^{-1})}{\sqrt{p+2k-2}}+\frac{4}{p+2k-2}\bigg\} \|h^{(k-1)}\|,  \quad k\geq 1,
\end{equation}
where $h^{(0)}\equiv h$. The conditions on the test function $h$ for the bounds (\ref{lukluk2}) and (\ref{gamma22}) are the same (but presented in a simpler manner) as those used by \cite{gaunt chi square}, in which it was noted that the bounds are valid under these conditions that are weaker than those presented in \cite{luk} and \cite{gaunt thesis}. From (\ref{gamma22}), we deduce (using that $p\geq1$ and $k\geq2$) the simplified bounds
\begin{equation}\label{63 bound}\|f^{(k)}\|\leq\frac{6.375\|h^{(k-1)}\|}{\sqrt{p+1}}, \quad k\geq2,
\end{equation}
and
\begin{equation}\label{632 bound}\|f^{(k)}\|\leq \bigg\{\frac{\sqrt{2}(\sqrt{2\pi}+\mathrm{e}^{-1})}{\sqrt{2k-1}}+\frac{4}{2k-1}\bigg\}\|h^{(k-1)}\|, \quad k\geq2.
\end{equation}
Finally, we record a bound of \cite{gaunt chi square}. For $h\in C_b^{k-2,k-1}(\mathbb{R}^+)$,
\begin{align}\label{useful1}\|f^{(k)}\|\leq
\frac{4}{p+2}\big(3\|h^{(k-1)}\|+2\|h^{(k-2)}\|\big), \quad k\geq 2.
\end{align}

\vspace{3mm}

\noindent{\emph{Proof of Theorem \ref{thm0}.}} In the proof of the bounds (\ref{pearbound1}) and (\ref{pearbound01}), \cite{gaunt chi square} obtained the intermediate bounds
\begin{align}|\mathbb{E}[h(\chi^2)]-\chi_{(r-1)}^2h|&\leq \frac{1}{n}\{19\|f''\|+309\|f^{(3)}\|+1089\|f^{(4)}\|\nonumber\\
\label{int111}&\quad+1997\|f^{(5)}\|+100\,974\|f^{(6)}\|\}\bigg(\sum_{j=1}^r\frac{1}{\sqrt{p_j}}\bigg)^2,
\end{align}
and
\begin{equation}\label{int222}|\mathbb{E}[h(\chi^2)]-\chi_{(r-1)}^2h|\leq \frac{1}{\sqrt{n}}\{18\|f''\|+84\|f^{(3)}\|\}\sum_{j=1}^r\frac{1}{\sqrt{p_j}},
\end{equation}
where $f$ is the solution (\ref{chisteinsoln}) of the $\chi_{(r-1)}^2$ Stein equation. 

We note that bounding the derivatives of the solution $f$ of the $\chi_{(r-1)}^2$ Stein equation in (\ref{int111}) and (\ref{int222}) using inequality (\ref{useful1}) then yields the bounds (\ref{pearbound1}) and (\ref{pearbound01}), respectively.
The bounds (\ref{pearbound1bb}) and (\ref{pearbound01bb}) of Theorem \ref{thm0} are obtained by instead using inequality (\ref{63 bound}) to bound the derivatives of $f$ in the bounds (\ref{pearbound1}) and (\ref{pearbound01}), respectively.  To get the bound (\ref{pearbound1cc}) of Theorem \ref{thm0}, we use inequality (\ref{lukluk2}) to bound $\|f''\|,\ldots,\|f^{(5)}\|$ and (\ref{632 bound}) to bound $\|f^{(6)}\|$ in (\ref{int111}). Finally, to obtain the bound  (\ref{pearbound01cc}) we use (\ref{lukluk2}) to bound $\|f''\|$ and (\ref{632 bound}) to bound $\|f^{(3)}\|$ in (\ref{int222}). \hfill $\Box$

\section{Proof of Lemma \ref{lem1}}\label{appa}

\noindent{\emph{Proof of Lemma \ref{lem1}.}} (i) Without loss of generality, we let $a=1$; the general $a>0$ case follows by rescaling. We therefore need to prove that, for $x\geq0$,
\begin{equation}\label{lembound1b}|f(x)|\leq \frac{2}{3}(x-1)^4,
\end{equation}
where
\begin{equation*}f(x):=2x\log(x)-2(x-1)-(x-1)^2+\frac{1}{3}(x-1)^3.
\end{equation*}
It is readily checked that inequality (\ref{lembound1b}) holds for $x=0$ and $x=2$. For $0<x<2$ (that is $|x-1|<1$), we can use a Taylor expansion to obtain the bound
\begin{equation*}|f(x)|=2(x-1)^4\bigg|\sum_{k=0}^\infty\frac{(-1)^k(x-1)^k}{(k+3)(k+4)}\bigg|\leq2(x-1)^4\sum_{k=0}^\infty\frac{1}{(k+3)(k+4)}=\frac{2}{3}(x-1)^4,
\end{equation*}
so inequality (\ref{lembound1b}) is satisfied for $0<x<2$. Now, suppose $x>2$. We have that $f'(x)=2\log(x)-2(x-1)+(x-1)^2$ and $\frac{\mathrm{d}}{\mathrm{d}x}\big(2(x-1)^4/3\big)=8(x-1)^3/3$. By the elementary inequality $\log(u)\leq u-1$, for $u\geq1$, we get that
\begin{align*}|f'(x)|&=|2\log(x)-2(x-1)+(x-1)^2|\\
&=2\log(x)-2(x-1)+(x-1)^2\leq (x-1)^2\leq \frac{8}{3}(x-1)^3,
\end{align*}
where the final inequality holds because $x>2$. Therefore, for $x>2$, $2(x-1)^4/3$ grows faster than $|f(x)|$. Since $|f(2)|<2(2-1)^4/3=2/3$, it follows that inequality (\ref{lembound1b}) holds for all $x>2$. We have now shown that  (\ref{lembound1b}) is satisfied for all $x\geq0$, as required. 

\vspace{2mm}

\noindent{(ii)} Suppose $\lambda\geq3$. Again, without loss of generality, we may set $a=1$. We therefore need to prove that, for $x\geq0$,
\begin{equation}\label{lembound3b}|g_\lambda(x)|\leq \frac{(\lambda-2)(\lambda-1)\lambda(\lambda+1)}{24}\big(1+x^{\lambda-3}\big)(x-1)^4,
\end{equation}
where
\begin{equation}\label{gglll}g_\lambda(x):=x^{\lambda+1}-1-(\lambda+1)(x-1)-\frac{\lambda(\lambda+1)}{2}(x-1)^2-\frac{(\lambda-1)\lambda(\lambda+1)}{6}(x-1)^3.
\end{equation}
By a Taylor expansion of $x^{\lambda+1}$ about $x=1$ we have that
\begin{equation}\label{subglam}g_\lambda(x)
=\frac{(\lambda-2)(\lambda-1)\lambda(\lambda+1)}{24}\xi^{\lambda-3}(x-1)^4,
\end{equation}
where $\xi>0$ is between 1 and $x$. Now, as $\xi$ is between $1$ and $x$ and because $\lambda\geq3$, we have that
\begin{equation*}\xi^{\lambda-3}\leq(\mathrm{max}\{1,x\})^{\lambda-3}
\leq 1+x^{\lambda-3},
\end{equation*}
and applying this inequality to (\ref{subglam}) gives us (\ref{lembound3b}), as required.

\vspace{2mm}

\noindent{(iii)} Suppose now that $\lambda\in(-1,3)\setminus\{0\}$.  Without loss of generality, we set $a=1$, and it therefore suffices to prove that, for $x\geq0$,
\begin{equation}\label{part333} |g_\lambda(x)|\leq \frac{|(\lambda-2)(\lambda-1)\lambda|}{6}(x-1)^4.
\end{equation}
 We shall verify inequality (\ref{part333}) by treating the cases $0< x\leq 2$ and $x\geq2$ separately (it is readily checked that the inequality holds at $x=0$).  For $0<x<2$ (that is $|x-1|<1$) we can use a Taylor expansion to write 
\begin{equation*}g_\lambda(x)=(x-1)^4 G_\lambda(x),
\end{equation*}
where
\begin{equation*}G_\lambda(x)=\sum_{k=0}^\infty\binom{\lambda+1}{k+4}(x-1)^k,
\end{equation*}
and the generalised binomial coefficient is given by $\binom{a}{j}=[a(a-1)(a-2)\cdots(a-j+1)]/j!$, for $a>0$ and $j\in\mathbb{N}$. We now observe that, since $\lambda<3$, the generalised binomial coefficients $\binom{\lambda+1}{k+4}$ are either positive for all even $k\geq0$ and negative for all odd $k\geq1$, or are negative for all even $k\geq0$ and positive for all odd $k\geq1$ (or, exceptionally always equal to zero if $\lambda\in\{1,2\}$, which is a trivial case in which $g_\lambda(x)=0$ for all $x\geq0$). Hence, for $0<x<2$, $G_\lambda(x)$ is bounded above by $|G_\lambda(0)|$, and a short calculation using the expression (\ref{gglll}) (note that $G_\lambda(x)=g_\lambda(x)/(x-1)^4$) shows that $G_\lambda(0)=|(\lambda-2)(\lambda-1)\lambda|/6$. Thus, for $0\leq x<2$, we have the bound
\begin{equation*}|g_\lambda(x)|
\leq\frac{|(\lambda-2)(\lambda-1)\lambda|}{6}(x-1)^4.
\end{equation*}

Suppose now that $x\geq2$. Recall from (\ref{subglam}) that
\begin{equation*}g_\lambda(x)
=\frac{(\lambda-2)(\lambda-1)\lambda(\lambda+1)}{24}\xi^{\lambda-3}(x-1)^4,
\end{equation*}
where $\xi>0$ is between 1 and $x$. In fact, because we are considering the case $x\geq2$, we know that $\xi>1$.  Therefore, since $\lambda<3$, we have that $\xi^{\lambda-3}<1$. Therefore, for $x\geq2$,
\begin{equation*}|g_\lambda(x)|
=\frac{|(\lambda-2)(\lambda-1)\lambda(\lambda+1)|}{24}(x-1)^4\leq\frac{|(\lambda-2)(\lambda-1)\lambda|}{6}(x-1)^4,
\end{equation*}
where the second inequality follows because $\lambda\in(-1,3)\setminus\{0\}$. We have thus proved inequality (\ref{part333}), which completes the proof of the lemma. \hfill $\Box$

\section*{Acknowledgements}
The author is supported by a Dame Kathleen Ollerenshaw Research Fellowship. Preliminary research on this project was carried out whilst the author was supported by EPSRC grant EP/K032402/1. I would like to thank the referee for their helpful comments and suggestions.

\section{Supplementary Information}

\noindent{\emph{Proof of Lemma \ref{lem2}.}} (i)  Without loss of generality, we let $a=1$; the general $a>0$ case follows by rescaling. We therefore need to prove that, for $x\geq0$,
\begin{equation}\label{lembound2b}|f(x)|\leq |x-1|^3,
\end{equation}
where
\begin{equation*}f(x):=2x\log(x)-2(x-1)-(x-1)^2.
\end{equation*}
It is readily checked that inequality (\ref{lembound2b}) holds for $x=0$ and $x=2$. For $0<x<2$ (that is $|x-1|<1$), we can use a Taylor expansion to obtain the bound
\begin{equation*}|f(x)|=2|x-1|^3\bigg|\sum_{k=0}^\infty\frac{(-1)^k(x-1)^k}{(k+2)(k+3)}\bigg|\leq2|x-1|^3\sum_{k=0}^\infty\frac{1}{(k+2)(k+3)}=|x-1|^3,
\end{equation*}
so inequality (\ref{lembound2b}) is satisfied for $0<x<2$. Now, suppose $x>2$. We have that $f'(x)=2\log(x)-2(x-1)$ and $\frac{\mathrm{d}}{\mathrm{d}x}\big((x-1)^3\big)=3(x-1)^2$. By the inequality $\log(u)\leq u-1$, for $u\geq1$, we get that
\begin{align*}|f'(x)|&=|2\log(x)-2(x-1)|=2(x-1)-2\log(x)\leq (x-1)^2\leq 3(x-1)^2,
\end{align*}
where the final inequality holds because $x>2$. Therefore, for $x>2$, $(x-1)^3$ grows faster than $|f(x)|$. Since $|f(2)|=(2-1)^3=1$, it follows that inequality (\ref{lembound2b}) holds for all $x>2$. We have now shown that inequality (\ref{lembound2b}) is satisfied for all $x\geq0$, as required. 

\vspace{2mm}

\noindent{(ii)} Again, without loss of generality, we may set $a=1$. We therefore need to prove that, for $x\geq0$,
\begin{equation}\label{lembound3bu}|g_\lambda(x)|\leq \frac{(\lambda-1)\lambda(\lambda+1)}{6}\big(1+x^{\lambda-2}\big)|x-1|^3,
\end{equation}
where
\begin{equation}\label{gglllu}g_\lambda(x):=x^{\lambda+1}-1-(\lambda+1)(x-1)-\frac{\lambda(\lambda+1)}{2}(x-1)^2.
\end{equation}
By a Taylor expansion of $x^{\lambda+1}$ about $x=1$ we have that
\begin{equation}\label{subglamu}g_\lambda(x)
=\frac{(\lambda-1)\lambda(\lambda+1)}{6}\xi^{\lambda-2}(x-1)^3,
\end{equation}
where $\xi>0$ is between 1 and $x$. Now, as $\xi$ is between $1$ and $x$ and because $\lambda\geq2$, we have that
\begin{equation*}\xi^{\lambda-2}\leq(\mathrm{max}\{1,x\})^{\lambda-2}
\leq 1+x^{\lambda-2},
\end{equation*}
and applying this inequality to (\ref{subglamu}) gives us (\ref{lembound3bu}), as required.

\vspace{2mm}

\noindent{(iii)} Suppose now that $\lambda\in(-1,2)\setminus\{0\}$.  Without loss of generality, we set $a=1$, and it therefore suffices to prove that, for $x\geq0$,
\begin{equation}\label{part333u} |g_\lambda(x)|\leq \frac{|(\lambda-1)\lambda|}{2}|x-1|^3.
\end{equation}
 We shall verify inequality (\ref{part333u}) by treating the cases $0< x\leq 2$ and $x\geq2$ separately (it is readily checked that the inequality holds at $x=0$).  For $0<x<2$ (that is $|x-1|<1$) we can use a Taylor expansion to write 
\begin{equation*}g_\lambda(x)=(x-1)^3 G_\lambda(x),
\end{equation*}
where
\begin{equation*}G_\lambda(x)=\sum_{k=0}^\infty\binom{\lambda+1}{k+3}(x-1)^k,
\end{equation*}
and the generalised binomial coefficient is given by $\binom{a}{j}=[a(a-1)(a-2)\cdots(a-j+1)]/j!$, for $a>0$ and $j\in\mathbb{N}$. We now observe that, since $\lambda<2$, the generalised binomial coefficients $\binom{\lambda+1}{k+3}$ are either positive for all even $k\geq0$ and negative for all odd $k\geq1$, or are negative for all even $k\geq0$ and positive for all odd $k\geq1$ (or, exceptionally always equal to zero if $\lambda=1$, which is a trivial case in which $g_\lambda(x)=0$ for all $x\geq0$). Hence, for $0<x<2$, $G_\lambda(x)$ is bounded above by $|G_\lambda(0)|$, and a short calculation using the expression (\ref{gglllu}) (note that $G_\lambda(x)=g_\lambda(x)/(x-1)^3$) shows that $G_\lambda(0)=|(\lambda-1)\lambda|/2$. Thus, for $0\leq x<2$, we have the bound
\begin{equation*}|g_\lambda(x)|
\leq\frac{|(\lambda-1)\lambda|}{2}|x-1|^3.
\end{equation*}

Suppose now that $x\geq2$. Recall from (\ref{subglamu}) that
\begin{equation*}g_\lambda(x)
=\frac{(\lambda-1)\lambda(\lambda+1)}{6}\xi^{\lambda-2}(x-1)^3,
\end{equation*}
where $\xi>0$ is between 1 and $x$. In fact, because we are considering the case $x\geq2$, we know that $\xi>1$.  Therefore, since $\lambda<2$, we have that $\xi^{\lambda-2}<1$. Therefore, for $x\geq2$,
\begin{equation*}|g_\lambda(x)|
=\frac{|(\lambda-1)\lambda(\lambda+1)|}{6}|x-1|^3\leq\frac{|(\lambda-1)\lambda|}{2}|x-1|^3,
\end{equation*}
where the second inequality follows because $\lambda\in(-1,2)\setminus\{0\}$. We have thus proved inequality (\ref{part333u}), which completes the proof of the lemma. \hfill $\Box$


\begin{thebibliography}{99}
\bibitem{pearson} Pearson, K.  On the criterion that a given system of deviations is such that it can be reasonably supposed to have arisen from random sampling.  \emph{Phil. Mag.} $\mathbf{50}$ (1900),  157--175.
\end{thebibliography}
\end{document}